\documentclass[10pt]{article}
\usepackage[all]{xy}
\usepackage{amsfonts,amsmath,oldgerm,amssymb,amscd}
\newcommand{\ra}{\rightarrow}		
\newcommand{\lra}{\longrightarrow}

\newcommand{\surj}{\ra\!\!\!\ra}	

\newcommand{\ol}{\overline}		
\newcommand{\wt}{\widetilde}

\newtheorem{theorem}{Theorem}[section]
\newtheorem{proposition}[theorem]{Proposition}
\newtheorem{lemma}[theorem]{Lemma}
\newtheorem{definition}[theorem]{Definition}
\newtheorem{corollary}[theorem]{Corollary}

\newcommand{\ga}{\alpha}	\newcommand{\gb}{\beta}
		\newcommand{\gd}{\delta}
	
	\newcommand{\gl}{\lambda}
	\newcommand{\gp}{\phi}
		\newcommand{\gs}{\sigma}
\newcommand{\gt}{\theta}

	\newcommand{\gD}{\Delta}
	
	\newcommand{\gT}{\Theta}

\newcommand{\BA}{\mbox{$\mathbb A$}}	
\newcommand{\BC}{\mbox{$\mathbb C$}}

	\newcommand{\BP}{\mbox{$\mathbb P$}}
\newcommand{\bq}{\mbox{$\mathbb Q$}}

\newcommand{\CE}{\mbox{$\mathcal E$}}	
	
\newcommand{\CI}{\mbox{$\mathcal I$}}	\newcommand{\CJ}{\mbox{$\mathcal J$}}
\newcommand{\CK}{\mbox{$\mathcal K$}}

\newcommand{\ot}{\mbox{\,$\otimes$\,}}	\newcommand{\op}{\mbox{$\oplus$}}
\newcommand{\Spec}{\mbox{\rm Spec\,}}	\newcommand{\hh}{\mbox{\rm ht\,}}
	
\newcommand{\Aut}{\mbox{\rm Aut\,}}	
	
\newcommand{\dd}{\mbox{\rm dim\,}}	

\newcommand{\Um}{\mbox{\rm Um}}

\newcommand{\sur}{\twoheadrightarrow}
\newcommand{\bp}{\begin{proposition}}
\newcommand{\ep}{\end{proposition}}
\newcommand{\bl}{\begin{lemma}}
\newcommand{\el}{\end{lemma}}
\newcommand{\bt}{\begin{theorem}}
\newcommand{\et}{\end{theorem}}
\newcommand{\bc}{\begin{corollary}}
\newcommand{\ec}{\end{corollary}}
\newcommand{\bd}{\begin{definition}}
\newcommand{\ed}{\end{definition}}

\def\rmk{\refstepcounter{theorem}\paragraph{{\bf Remark} \thetheorem}}
\def\proof{\paragraph{Proof}}
\def\example{\refstepcounter{theorem}\paragraph{{\bf Example} \thetheorem}}
\def\quest{\refstepcounter{theorem}\paragraph{{\bf Question} \thetheorem}}
\def\definition{\refstepcounter{theorem}\paragraph{{\bf Definition} \thetheorem}}
\newcommand{\remark}{\rmk}
\begin{document}

\begin{center}
{\Large \bf A question of Nori : Further studies and applications} 
\end{center}

\begin{center}
{\large Mrinal Kanti Das and Manoj K. Keshari}
\end{center}

\begin{center}
Stat-Math Unit, Indian Statistical Institute,
 203 B. T. Road, Kolkata 700108; mrinal@isical.ac.in

Department  of Mathematics,
 IIT Bombay, Mumbai 400076; keshari@math.iitb.ac.in
\end{center}

\section{Introduction}
Let $R$ be a commutative, Noetherian ring and $I\subset R$ be an ideal
such that $I$ is locally generated by $n$ elements. In general, local
generators may not have a lift to a set of $n$ global generators of
$I$.  For example, if $R$ is the coordinate ring of a real $3$-sphere
and $I=\mathfrak{m}$ is a real maximal ideal, then it is well-known
that $\mathfrak{m}/\mathfrak{m}^2$ is $3$-generated (therefore,
$\mathfrak{m}$ is locally $3$-generated) but no surjection
$\ol{\theta}: R^3 \surj \mathfrak{m}/\mathfrak{m}^2$ can be lifted to
a surjection $\theta :R^3\surj \mathfrak{m}$ (in fact, $\mathfrak{m}$
is not generated by $3$ elements).  Let us now cite a non-trivial
instance where one has an affirmative conclusion.  Let $R=A[T]$ and
$I\subset A[T]$ be an ideal such that $I/I^2$ is generated by $n$
elements, where $n\geq \dd(A[T]/I)+2$. If $I$ contains a monic
polynomial, then it is implicit in Mandal's famous result \cite{m1}
that any surjection $\ol{\theta}: A[T]^n\sur I/I^2$ can be lifted to a
surjection $\theta : A[T]^n\sur I$.

We now focus on the set up when $R=A[T]$ and consider the following
(more general) situation : Let $I\subset A[T]$ be an ideal and $P$ be
a projective $A$-module of rank $n\geq \dd(A[T]/I)+2$. Assume that
there is a surjection $\ol{\theta}: P[T]\surj I/I^2$. One may wonder,
under what condition(s) $\ol{\theta}$ may be lifted to a surjection
$\theta:P[T]\surj I$. A necessary condition obviously would be that
$I(0)$ should be image of $P$. But there are examples \cite[5.2]{brs1}
to show that this is not sufficient.  In this context, an intriguing
open question is the following one which is a variant of a question of
Nori (see \cite{m2,mrs}).

\medskip

\quest\label{open} 
Let $A$ be a regular ring, $I\subset A[T]$ be an
ideal and $P$ be a projective $A$-module of rank $n$ where $n\geq \dd
(A[T]/I)+2$. Assume that there is a surjection $\ol{\varphi}:P[T]\surj
I/(I^2T)$. Can $\ol{\varphi}$ be lifted to a surjective $A[T]$-linear
map $\varphi: P[T]\surj I$?

\medskip

Note that $\ol{\varphi}(0):P\surj I(0)$.
The assumption of regularity cannot be dropped (as there is an example
due to Bhatwadekar, Mohan Kumar and Srinivas \cite[6.4]{brs1}) unless
the ideal $I$ has some special properties (as shown by Mandal in
\cite{m2} where he gave an affirmative answer when $I$ contains a monic
polynomial, without assuming the ring to be regular). The best result
that we have so far is the one due to Bhatwadekar-Keshari
\cite[4.13]{bk} where they assume $A$ to be a regular domain of
dimension $d$ which is essentially of finite type over an infinite
perfect field $k$ and further that $\hh I=n$ with $2n\geq d+3$. It has
been shown in \cite{d4} that one need not take the field to be perfect
in \cite[4.13]{bk}.

As it stands, even when $A$ is a ring of the type as considered in
\cite{bk}, one does not have a complete answer to (\ref{open}). One of our objectives
in this paper is to explore whether we can improve the result of
\cite{bk} by relaxing the condition on height of the ideal, i.e.,
allowing ideals of possibly smaller height into consideration. We
prove (see (\ref{4.13}) below)

\begin{theorem}
Let $A$ be a regular domain which is essentially of finite type over
an infinite field $k$. Let $I\subset A[T]$ be an ideal and $P$ be a
projective $A$-module of rank $n$ where $n+\hh I \geq \dd
A[T]+2$. Then any surjection $\phi: P[T] \surj I/(I^2T)$ can be lifted
to a surjection from $P[T]$ to $I$.
\end{theorem}

Note that when $\hh I=n$, we recover \cite[4.13]{bk}. However, we are
not giving a new proof of the theorem of  Bhatwadekar-Keshari here. We are using
their theorem to prove the above result, for which we need to have a
suitable set of so called ``moving lemma", ``addition" and
``subtraction" principles in a more general set up than the existing
ones. We prove these in Section 3.

The cluster of moving lemma, addition and subtraction principles also
enables us to answer another interesting question. Let $R$ be a
Noetherian ring, $J\subset R$ be an ideal with $\mu(J/J^2)=n$, where
$n+\hh J \geq \dd R+3$. Given a surjection $\omega_J:(R/J)^n\surj
J/J^2$ we associate an element $s^n(J,\omega_J)$ in the $n^{th}$ Euler
class group $E^n(R)$ (see (\ref{loweuler}) for the definition of
$E^n(R)$) such that $s^n(J,\omega_J)=0$ if and only if $\omega_J$ can
be lifted to a surjective map $\theta : R^n\surj J$. We call
$s^n(J,\omega_J)$ the $n^{th}$ Segre class of the pair $(J,\omega_J)$.
For further details look at Section 4.

Any reader familiar with the theory of Euler class groups will be
aware that Question \ref{open} is intimately related to the Euler
class groups. When Nori proposed the definition of the Euler class
group (of a smooth affine domain $A$) some twenty years back, he also
suggested Question \ref{open} as an important tool (see \cite{m2,mrs}
for motivation). Nori's definition of the Euler class group, given in
terms of ``homotopy" with respect to the affine line $\BA^1$, appeared
in \cite{brs1} (see (\ref{nori}) below). However,
Bhatwadekar-Sridharan settled (\ref{open}) in some special case
\cite[3.8]{brs1}, then came up with an alternative definition of the
Euler class group and proved its equivalence with Nori's definition by
using \cite[3.8]{brs1}. Section 5 is about revival of Nori's
``homotopical" definition of the Euler class group. In this section we
closely investigate Nori's definition, illustrate how it works for
smooth affine domains and recover the main result of \cite{brs1} on
Euler classes. With an example due to Bhatwadekar (personal
communication) we show why Nori's definition does not naturally extend
to non-regular rings. At the end, we formulate the correct
generalization to such rings.


\section{Preliminaries}
All rings are assumed to be commutative Noetherian and modules are
assumed to be finitely generated. We refer to \cite{bk} for any
undefined term. In this section we collect some results which will be
used frequently in later sections.

We begin by stating two classical results due to Serre \cite{s}
and Bass \cite{Bass},  respectively.

\begin{theorem}\label{serre} (Serre) 
Let $A$ be a ring and $P$ be a projective $A$-module. If
$\text{rank}(P) >\dd A/\CJ(A)$, then $P\simeq Q\oplus A$ for some
$A$-module $Q$, where $\CJ(A)$ is the Jacobson radical of $A$.
\end{theorem}

\begin{theorem}\label{bass} (Bass) 
Let $A$ be a ring and let $P$ be a projective $A$-module of rank $>\dd
A/\CJ(A)$. Then the group ${\mathcal E}(P\op A)$ of transvections of
$P\op A$ acts transitively on $\Um(P\op A)$.
\end{theorem}

The following result is about lifting of automorphisms \cite[4.1]{br}.

\begin{proposition}\label{b-roy}
Let $I$ be an ideal of a ring $A$ and let $P$ be a projective
$A$-module. Then any transvection $\phi$ of $P/IP$ (i.e. $\phi\in
{\mathcal E}(P/IP)$) can be lifted to an automorphism $\Phi$ of $P$.
\end{proposition}

The following result is due to Eisenbud and Evans \cite{ee}.

\begin{theorem}\label{EE}
Let $A$ be a ring and let $P$ be a projective $A$-module of rank
$n$. Let $(\ga,a)\in (P^*\op A)$. Then there exists $\gb\in P^*$ such
that $\hh I_a \geq n$, where $I=(\ga+a\gb)(P)$. In
particular, if the ideal $(\ga(P),a)$ has height $\geq n$ and $I$ is a
proper ideal of $A$, then $\hh I =n$.
\end{theorem}

The next two results are due to Bhatwadekar and Raja Sridharan.

\begin{lemma}\label{2.11} \cite[2.11]{brs2}
Let $A$ be a ring and let $I$ be an ideal of
$A$.  Let $I_1$ and $I_2$ be ideals of $A$ contained in $I$ such that
$I_2\subset I^2$ and $I_1+I_2=I$.  Then $I=I_1+(e)$ for some $e\in
I_2$ and $I_1=I\cap I'$, where $I_2+I'=A$.  
\end{lemma}

\begin{lemma}\label{BRS}\cite[3.5]{brs1}
Let $A$ be a regular domain containing a field and let $I$ be an ideal
of $A[T]$. Let $P$ be a projective $A$-module and $J=I\cap A$. Let
$\phi :P[T] \surj I/(I^2T)$ be a surjection. Suppose $\phi\ot
A_{1+J}[T]$ can be lifted to a surjection $\psi :P_{1+J}[T]\surj
I_{1+J}$. Then $\phi$ can be lifted to a surjection $\Phi :P[T]\surj
I$.
\end{lemma}

The following result is due to Mandal and Raja Sridharan \cite[2.3]{mrs}.

\begin{theorem}\label{mrs}
Let $A$ be a ring and let $I_1,I_2$ be two comaximal ideals of
$R=A[T]$. Assume that $I_1$ contains a monic polynomial and $I_2$ is
an extended ideal from $A$, i.e. $I_2=I_2(0)R$. Let $P$ be a
projective $A$-module of rank $r \geq \dd (R/I_1) +2$ and let
$I=I_1\cap I_2$. Let $\rho : P\surj I(0)$ and $\gd : P[T]/I_1P[T]
\surj I_1/I_1^2$ be two surjections such that $\gd(0)=\rho \ot
A/I_1(0)$. Then there exists a surjection $\eta:P[T] \surj I$ such
that $\eta(0)=\rho$.
\end{theorem}


\section{Main Theorem}
{\bf Unless otherwise mentioned, by a ring we mean a commutative
Noetherian ring}.

We  begin with a lemma. When $\hh I=n$ and $f=1$, it
follows from \cite[5.5]{bk}.

\begin{lemma}\label{moving} 
(Moving lemma) Let $J$ be an ideal of a ring $A$ and let
$P$ be a projective $A$-module of rank $n \geq \dd A/J +1$. Let
$\gt:P \surj J/J^2f$ be a surjection for some $f\in A$. Given any
ideal $K \subset A$ with $\dd A/K \leq n-1$, $\gt$ can be
lifted to a surjection $\gT : P \surj J''$ such that
 
$(i)$ $J'' + (J^2\cap K)f=J$,

$(ii)$ $J''=J\cap J'$, where $\hh J' \geq n$ and

$(iii)$ $J' +(J^2\cap K)f=A$.
\end{lemma}

\proof
Let $\gD : P\ra J$ be a lift of $\gt$. Then $\gD(P)+J^2f=J$. By
(\ref{2.11}), there exists $b\in J^2f$ such that $\gD(P)+(b)=J$. Let
``bar" denote reduction modulo the ideal $(J^2\cap K)f$. Note that
$\dd A/(J^2\cap K) \leq n-1$.

Applying  (\ref{EE}) on $(\ol \gD, \ol b)\in 
(\ol {P^*}\op \ol A)$, there exists $\gD_1 \in { P^*}$ such that if
$N=(\gD+b\gD_1)(P)$, then $\hh \ol N_{\ol b} \geq n$.  Since
$\gD\op b\gD_1$ is also a lift of $\gt$, replacing $\gD$ with
$\gD+b\gD_1$, we may assume that $N=\gD(P)$.

Now $N+(b)=J$ and $b\in J^2f$. By (\ref{2.11}), $N=J \cap J_1$, where
$J_1+(b)=A$.  Since $N_b=(J_1)_b$ and $\ol N=\ol J\cap \ol J_1$, we get $\hh
\ol J_1 = \hh  (\ol J_1)_{\ol b} = \hh \ol N_{\ol b} \geq n$. But
$n\leq \hh \ol J_1 = \hh (\ol J_1)_{\ol f} \leq \dd \ol {A_f} \leq
n-1$. Hence we get $\ol J_1=\ol A$. Therefore, $\ol N=\ol J$,
i.e.  $\gD(P)+(J^2\cap K)f =J$.

By (\ref{2.11}), there exists $c\in (J^2\cap K)f$ such that
$\gD(P)+(c)=J$. By (\ref{EE}), replacing $\gD$ by $\gD+c\gD_2$ for
some $\gD_2\in P^*$, we may assume that $\gD(P)=J\cap J'$, where
$\hh J' \geq n$ and $J'+(c)=A$. 
This proves the lemma.
$\hfill \square$

We now prove some ``addition" and ``subtraction" principles tailored
to suit our needs.

\begin{proposition}\label{sub}
(Subtraction Principle) Let $I,J$ be two comaximal ideals of a ring
$A$ and let $P=Q\op A$ be a projective $A$-module of rank $n$. Assume
that $n\geq \dd (A/J) +2$ and $n+ \hh I \geq \dd A+3$.  Assume that
$\Phi:P\surj I$ and $\Psi :P \surj I\cap J$ are two surjections such
that $\Phi\ot A/I=\Psi\ot A/I$. Then there exists a surjection
$\gD:P\surj J$ such that $\gD\ot A/J=\Psi \ot A/J$.
\end{proposition}

\begin{proof}
As each of $I,J$ is locally generated by $n$ elements, we have $\hh
I\leq n$ and $\hh J\leq n$. Note that to prove the result we can
change $\Phi$ and $\Psi$ by composing it with automorphisms of $Q\op A$.

Let ``bar'' denote reduction modulo $J^2$ and we write $\Phi =
(\Phi_1,a_1)$, where $\Phi_1 \in Q^*$.  Then $(\ol \Phi_1,\ol a_1) \in
\Um(\ol Q\op \ol A)$. Since $\dd \ol A \leq n-2$, by 
(\ref{serre}), $\ol Q$ has a unimodular element i.e. $\ol Q=Q_1\op \ol
A$. Write $\ol \Phi_1 \in {\ol Q}^*$ as $(\ol \ga,\ol b_1)$, where $\ol \ga \in
Q_1^*$. By (\ref{bass}), there exists $\sigma \in {\mathcal E}(\ol Q\op
\ol A)$ such that $(\ol \ga,\ol b_1,\ol a_1)\sigma=(0,1,0)$.

Using (\ref{b-roy}), let $\gt \in \Aut (Q\op A)$ be a lift of
$\sigma$.  If $(\Phi_1,a_1)\gt =(\Phi_2,a_2)$, then $a_2\in J^2$ and
$\ol \Phi_2 \in \Um(\ol Q)$.  By (\ref{EE}), there exists
$\Gamma\in Q^*$ such that if $K=(\Phi_2+a_2\Gamma)(Q)$, then $\hh
K_{a_2} \geq n-1$. Since $(\Phi_2+a_2\Gamma,a_2)$ is also a lift of
$\Phi=(\Phi_1,a_1)$, replacing $\Phi_2$ with $\Phi_2+a_2\Gamma$, we
may assume that $K=\Phi_2(Q)$ and $\hh K_{a_2} \geq n-1$. 
Note that $(K,a_2)=I$.

\smallskip

\noindent{\bf Case 1.} Assume that  $\hh I <n $. It is easy to see that
$\hh K =\hh I$. Since $\Phi_2(Q)+J^2=A$, replacing $a_2$ by
$a_2+\Phi_2(q)$ for some $q\in Q$, we may assume that $a_2=1$ modulo
$J^2$.

Consider the following ideals in the ring $A[Y]$: $\CI_1=(K,Y+a_2)$,
$\CI_2=JA[Y]$ and $\CI_3=\CI_1\cap \CI_2$. Note that $\CI_1(0) = I$
and $\CI_2(0)=J$. We have two surjections 
$$\Psi : P\surj \CI_3(0) ~~{\rm and}~~
\gd :=(\Phi_2 \ot A[Y], Y+a_2):P[Y] \surj \CI_1$$ 
such that $\Psi \ot
A/\CI_1(0)= \Phi\ot A/\CI_1(0)=\gd\ot A/\CI_1(0)$. Further $\dd
A[Y]/\CI_1 =\dd A/K$. 

Since $K\subset I$ have  the same height and $n+\hh
I \geq \dd A+3$, we have $\dd A/K \leq n-3$. Hence, applying
 (\ref{mrs}), we get a surjection $\eta:P[Y]
\surj \CI_3$ such that $\eta(0)=\Psi :Q\op A \surj J$. Since $1-a_2\in
J^2$, putting $Y=1-a_2$, we get a surjection $\eta(1-a_2):=\gD :P\surj
J$ with $\gD\ot A/J=\Psi \ot A/J$. This proves the result in this case.

\noindent{\bf Case 2.} Assume that $\hh I=n$. Then height of $K_{a_2} \geq
n-1$ and $I=(K,a_2)$ implies that $\hh K\geq n-1$. Since $n+\hh I\geq
\dd A+3$, we get $\dd A/K \leq n-2$. Now we can
complete the proof as in case 1.
$\hfill \square$ 
\end{proof}

\begin{proposition}\label{add}
(Addition principle) Let $I,J$ be two comaximal ideals of a ring $A$
and let $P=Q\op A$ be a projective $A$-module of rank $n$, where
$n+\hh (I\cap J)\geq \dd A+3$.  Let $\Phi:P\surj I$ and $\Psi: P\surj
J$ be two surjections. Then there exists a surjection $\gD:P\surj
I\cap J$ such that $\Phi\ot A/I=\gD\ot A/I$ and $\Psi\ot A/J =\gD\ot
A/J$.
\end{proposition}

\begin{proof}
As each of $I,J$ is locally generated by $n$ elements, we have $\hh
I\leq n$ and $\hh J\leq n$. Further, we can change $\Phi$ and $\Psi$
by composing it with automorphisms of $Q\op A$.

Let ``bar'' denote reduction modulo $J^2$ and write $B=A/J^2$. Since
$n+\hh J\geq \dd A+3$, we get $\dd B \leq n-3$, by 
(\ref{serre}), $\ol Q=Q_1\op B$. Further, $I+J=A$ and 
hence $\ol I=B$.  Write $\ol \Phi=(\phi_1,b_1,b_2):Q_1\op B^2 \surj B$
for the natural surjection induced from $\Phi$. Since $\ol \Phi
=(\phi_1,b_1,b_2)\in \Um (Q_1\op B^2)$, applying 
(\ref{bass}), we get $\Theta\in {\mathcal E}(Q_1\op B^2)$ such that
$(\phi_1,b_1,b_2)\Theta =(0,1,0)$. Using (\ref{b-roy}), let $\theta
\in {\mathcal E}(Q\op A)$ be a lift of $\Theta$. If $\Phi \theta =(\Phi_1,b)$,
then $\Phi_1(Q)+J^2=A$ and $b\in J^2$. Applying (\ref{EE}), and
replacing $\Phi_1$ by $\Phi_1+b \Gamma$ for some $\Gamma \in Q^*$, we
may assume that $\hh (H_b) \geq n-1$, where
$H=\Phi_1(Q)$. Since $n+\hh I\geq \dd A+3$, as in the proof of
(\ref{sub}), we can conclude that $(i)$ if $\hh I<n$, then $\hh I=\hh H$
and $(ii)$ if $\hh I=n$, then $\hh H \geq n-1$. In both the cases, we
get that $\dd A/H\leq n-2$.

Let $C=A/H$ and let ``tilde'' denote reduction modulo $H$. Since
$H+J^2=A$, we get $\wt \Psi\in \Um(\wt Q\op C)$. Further, $\dd C\leq
n-2$, hence by  (\ref{serre}), $\wt Q=Q_2\op C$. Write
$\wt \Psi =(\Psi_1,c_1,c_2) \in \Um(Q_2\op C^2)$. Applying 
 (\ref{bass}) on $\wt \Psi =(\Psi_1,c_1,c_2)$, we get $\Sigma
\in {\mathcal E}(Q_2\op C^2)$ such that $(\Psi_1,c_1,c_2)\Sigma=(0,1,0)$. Using
(\ref{b-roy}), let $\sigma \in {\mathcal E}(Q\op A)$ be a lift of $\Sigma$.  If
$\Psi \sigma=(\Psi_2,c)$, then $\Psi_2(Q)+ H=A$ and $c\in H$. Applying
(\ref{EE}), and replacing $\Psi_2$ by $\Psi_2+c\Gamma'$ for some
$\Gamma'\in Q^*$, we may assume that $\hh(K_c)\geq n-1$, 
where $K=\Psi_2(Q)$. Once again it is easy to see that
$(i)$ if $\hh J<n$, then $\hh K=\hh J$ and $(ii)$ if $\hh J=n$, then
$\hh K\geq n-1$.  From this, we can conclude that $\dd A/K\leq n-2$.
We have $H+K=A$ and $\dd A/(H\cap K) \leq n-2$. 

Consider the
following ideals of $A[T]$: 
$$\CI_1=(H,T^2+(b-2)T+1), \, 
\CI_2=(K,T^2+(c-2)T+1) ~~{\rm and}~~
\CI_3=\CI_1\cap \CI_2.$$ 
If we write $\gt_1=(\Phi_1, T^2+(b-2)T+1)$
and $\gt_2=(\Psi_2, T^2+(c-2)T+1)$, then $\gt_1,\gt_2$ are surjections
from $Q[T]\op A[T]$ to $\CI_1$ and 
$\CI_2$, respectively. Note that $\CI_1(1)=I$ and $\CI_2(1)=J$.

Since $\CI_1+\CI_2=A[T]$, we have $\CI_3/\CI_3^2=\CI_1/\CI_1^2\op
\CI_2/\CI_2^2$. Hence, using surjections $\gt_1$ and $\gt_2$, we get a
surjection $ \Delta : Q[T]\op A[T] \surj \CI_3/\CI_3^2$ such that 
$
\Delta \ot A[T]/\CI_1 = \gt_1 \ot A[T]/\CI_1$ and $ \Delta \ot
A[T]/\CI_2 = \gt_2 \ot A[T]/\CI_2.$ 
Since $\CI_3(0)=A$, it is easy to
see that the surjection $ \Delta$ can be lifted to a surjection
$\Delta_1:Q[T]\op A[T] \surj \CI_3/(\CI_3^2T)$. Since $\CI_3$ contains a monic
polynomial and $\dd A[T]/\CI_3=\dd A[T]/(\CI_1\cap \CI_2)= \dd A/(H\cap
K) \leq n-2$, applying  (\ref{mrs}), we can lift $\Delta_1$
to a surjection $\Delta_2:Q[T]\op A[T] \surj \CI_3$.

Write $\Delta_2(1):=\Delta_3$. Then $\Delta_3:Q\op A\surj I\cap J$ is
a surjection. Further, we have 
$\Delta_3 \ot A/I = \Delta(1) \ot A/I =
\gt_1(1)\ot A/I = (\Phi_1,b)\ot A/I = \Phi\ot A/I.$
Similarly, $\Delta_3\ot A/J=\Psi \ot A/J$. Hence $\Delta_3$ is the
required surjection. This completes the proof.  
$\hfill \square$ 
\end{proof}

\smallskip

The following result generalises \cite[4.6]{bk}. Recall that
$A(T)$ denotes the ring obtained from $A[T]$ by inverting all monic
polynomials.

\begin{lemma}\label{4.6}
Let $A$ be a ring, $I$ an ideal of $A[T]$  with $I+\CJ A[T]
=A[T]$ and let $n$ be a
positive integer such that $\dd A/\CJ\leq n-2$,
where $\CJ$ denotes the Jacobson radical of $A$. Let $P$ be
a projective $A$-module of rank $n\geq \dd A[T]-\hh I +2$.
Let $\phi : P[T] \surj I/I^2$ be a surjection. If
the surjection $\phi\ot A(T): P(T) \surj IA(T)/I^2A(T)$ can be lifted
to a surjection $\Phi_1 : P(T) \surj IA(T)$, then $\phi$ can be lifted
to a surjection $\Phi: P[T] \surj I$.
\end{lemma}

\begin{proof} 
As $I$ is locally generated by $n$ elements, we have $\hh I\leq
n$. Further, if $\hh I=n$, then the result follows from (\cite{bk},
Lemma 4.6). Hence, we assume that $\hh I <n$.

As $\dd A[T]/\CJ A[T] \leq 
n-1$, by (\ref{moving}), $\phi$ has a lift  $\Psi:
P[T] \surj J'$ such that 

$(i)$  $J'+(I^2\cap \CJ A[T])=I$, 

$(ii)$ $J'=I\cap J$, where $J$ is an ideal of height $\geq
n$, and 

$(iii)$ $J+(I^2\cap \CJ A[T])=A[T]$.

We assume that $\hh J=n$ (if $\hh J >n$ then $J=A$ and we are done). We
get a surjection $\psi : P[T] \surj J/J^2$ induced from $\Psi$. 
Note that $\Phi_1 \ot A(T)/IA(T) = \Psi\ot
A(T)/IA(T)$. Hence $\Phi_1$ is a lift of $\Psi\ot A(T)/IA(T)$.

We observe that $\dd A(T)-\hh IA(T)\leq \dd
A[T]-1-\hh I \leq n-3$ and

\begin{align*}
\dd (A(T)/JA(T))\leq \dd A(T)-\hh JA(T)\leq \dd A[T]-1-\hh J\\
= \dd A[T]-1-n \leq \dd A[T]-1-\hh I \leq n-3.
\end{align*}

Applying  (\ref{sub}) to the surjections $\Phi_1$ and $\Psi\ot A(T)$, we
get a surjection $\Psi_1 : P(T) \surj JA(T)$ such that $\Psi_1\ot
A(T)/JA(T) = \Psi \ot A(T)/JA(T) = \psi\ot A(T)/JA(T)$. 

Since $\hh J=n$ and $J+\CJ A[T]=A[T]$, applying \cite[4.6]{bk},  
we conclude that $\Psi\ot A[T]/J$ can be lifted to a surjection
$\gD:P[T] \surj J$. 
Note that $\dd A[T]/I \leq n-2$ and $\dd A[T]-\hh J < \dd A[T] -\hh
I \leq n-2$ (since $\hh J=n> \hh I$). Hence $\dd A[T]-\hh J \leq
n-3$.  Applying  (\ref{sub}) to
 $\gD$ and $\Psi$, we get a surjection $\Phi: P[T]
\surj I$ such that $\Phi \ot A[T]/I=\phi$. This proves the result.
$\hfill \square$ 
\end{proof}

\smallskip

The following result generalises (\cite{bk}, Proposition 4.9).

\begin{proposition}\label{4.9} 
Let $A$ be a regular domain containing a field and let $I$ be an ideal
of $A[T]$. Let $P$ be a projective $A$-module of rank
$n\geq \dd A[T]-\hh I+2$. Let $\psi: P[T] \surj I/I^2T$ be a
surjection. If there exists a surjection $\Psi' : P(T) \surj IA(T)$
which is a lift of $\psi \ot A(T)$, then we can lift $\psi$ to a
surjection from $P[T]$ to $I$.
\end{proposition}

\begin{proof} 
As $I$ is locally generated by $n$ elements, we have $\hh I\leq n$.
Further, if $\hh I=n$, then the result follows from \cite[4.9]{bk}.
 Hence, we assume that $\hh I<n$.

\noindent{\bf Step 1.}
By \cite[3.5]{brs1}, we may assume that $J=I\cap A\subset
\CJ(A)$. Note that  $\dd(A/J)\leq \dd A-\hh J=\dd
A[T]-1-\hh J\leq \dd A[T]-\hh I\leq n-2.$  Therefore, by (\ref{serre}), 
we may assume that $P=Q\op A^2$.

By (\ref{moving}), we can lift $\psi$ to a surjection $\Psi :P[T]
\surj I\cap I'$, where $I'\subset A[T]$ is of height $n$ with
$\Psi(P[T])+(J^2T)=I$ and $I'+(J^2T)=A[T]$. (If $\hh I'>n$, then
$I'=A[T]$ and we are done.)

Let $\psi_1 : P[T] \surj I'/{I'}^2$ be induced from
$\Psi$. Since $I'(0)=A$ and $P$ has a unimodular element, $\psi_1$ can
be lifted to a surjection $\psi_2 : P[T] \surj I'/{I'}^2T$ by \cite[3.9]{brs1}.

Observe that $\dd A(T) - \hh IA(T) \leq n-3$ and $\dd
A(T)/I'A(T) \leq n-2$.  Further, we have $\Psi \ot A(T)/IA(T)=\Psi'
\ot A(T)/IA(T)$. Hence, applying (\ref{sub}), we
get a surjection $\gD:P(T) \surj I'A(T)$ such that 
$ \gD \ot A(T)/I'A(T) =\psi_2\ot
A(T)/I'A(T) = \psi_1\ot A(T)/I'A(T).$ 
Since $\hh I'=n$, by
\cite[4.9]{bk}, $\psi_2$ can be lifted to a surjection
$\gD_1: P[T] \surj I'$. \\

\noindent{\bf Step 2.}  Write $B=A[T]/(J^2T)A[T]$. Since
$I'+(J^2T)=A[T]$, we get $(\gD_1\ot B)\in \Um({P[T]^*\ot B})$. Since
$P=Q\op A^2$, we write $\gD_1\ot B = ( \gD_2,a_1,a_2)$, where $\gD_2 \in
{Q[T]^*\ot B}$ and $a_1,a_2\in B$. Note that $ B/JB= (A/J)[T]$ and
$\dd A/J \leq n-2$.

Let ``bar'' denote reduction modulo $JB$ and write $\ol B:=B/JB$.  By
a result of
Plumstead
\cite{p}, there exists $\Theta \in {\mathcal E}(\ol {P[T]^*\ot B}) $ such that
$\Theta (\ol \gD_2,\ol a_1,\ol a_2)=(0,1,0)$. Since $JB$ is contained
in the Jacobson radical of $B$, we can lift $\Theta$ to $\Theta_1 \in
{\mathcal E}(P[T]^*\ot B)$ such that $\Theta_1(\gD_2,a_1,a_2)=(0,1,0)$. 
Let $\Theta_2\in \Aut (P[T]^*)$ be a lift of $\Theta_1$. If
$\Theta_2(\gD_1)=(\gD_3,b_1,b_2)$, then we get that
$\gD_3(Q[T])\subset (J^2T)$, $b_1=1$ modulo $(J^2T)$ and $b_2\in
(J^2T)$.

By (\ref{EE}), replacing $(\gD_3,b_1,b_2)$ by
$(\gD_3+b_2\gd_1,b_1+cb_2,b_2)$ for some $\gd_1 \in Q[T]^*$ and $c\in
A[T]$, we may assume that $\hh (\gD_3(Q[T]),b_1)=n-1$. Note that we
still have $(\gD_3(Q[T]),b_1)+(J^2T)=A[T]$. Further, replacing $b_2$
by $b_1+b_2$, we may assume that $b_2=1$ modulo $(J^2T)$. As
$(\gD_3(Q),b_1)$ is comaximal with $\CJ(A)A[T]$, we have

\begin{align*}
\dd A[T]/(\gD_3(Q[T]),b_1)\leq \dd A[T]-\hh (\gD_3(Q),b_1)-1\\
=\dd
A[T]-n\leq \dd A[T]-\hh I\leq n-2.
\end{align*}

Write $C:=A[T], \wt P=P[T]$ and consider the following ideals of
$C[Y]$: 
$$\CK_1=(\gD_3(Q[T]),b_1,Y+b_2), \,\CK_2=IC[Y] ~~{\rm and}~~
\CK_3=\CK_1\cap \CK_2.$$ 
Note that $\CK_1(0)=I'$ and $\CK_2(0)=I$. We
have two surjections 
$$\Psi : \wt P \surj \CK_3(0) ~~{\rm and}~~
\Gamma=(\gD_3,b_1,Y+b_2):\wt P[Y] \surj \CK_1$$ such that
$\Gamma(0)\ot C/\CK_1(0)=\Psi \ot C/\CK_1(0)$. Further,  $\dd
C[Y]/\CK_1=\dd C/(\gD_3(Q),b_1)\leq n-2$. Applying 
(\ref{mrs}), we get a surjection $\eta : \wt P[Y] \surj K_3$ such that
$\eta(0)=\Psi$. Since $1-b_2\in (J^2T)$, we get a surjection
$\eta_1:=\eta(1-b_2) : \wt P \surj I$ with $\eta_1=\Psi$ modulo
$(J^2T)$. This proves the result.  
$\hfill \square$ 
\end{proof}

\smallskip

We now prove our main theorem which is a generalization of \cite[4.13]{bk}.

\begin{theorem}\label{4.13} 
Let $k$ be an infinite  field and let $A$ be a
regular domain which is essentially of finite type
over $k$. Let $I$ be an ideal of $A[T]$ and let $P$
be a projective $A$-module of rank $n\geq \dd A[T]-\hh I+2$. 
Then any surjection $\phi: P[T] \surj I/(I^2T)$ can be lifted 
to a surjection from $P[T]$ to $I$.
\end{theorem}

\begin{proof}
As $I$ is locally generated by $n$ elements, we have $\hh I\leq
n$. Further, if $\hh I=n$, then the result follows from (\cite{bk},
Theorem 4.13). Hence, we assume that $\hh I<n$.

By (\ref{BRS}), we may assume that $J=I\cap A\subset \CJ(A)$. Since
$$\dd A/J \leq \dd A-\hh J \leq \dd A-\hh I+1 = \dd A[T]-\hh I\leq
n-2,$$ we may assume that $P$ has a unimodular element.

By (\ref{moving}), we can lift $\phi$ to a surjection $\Phi:P[T] \surj
I\cap I'$, where $I'$ is an ideal of $A[T]$ of height $n$ with
$\Phi(P[T])+(J^2T)=I$ and $I'+(J^2T)=A[T]$.
Let $\psi: P[T] \surj I'/{I'}^2$ be  induced from $\Phi$.
Since $I'(0)=A$ and $P$ has a unimodular element, $\psi$ can be lifted
to a surjection $\psi_1:P[T] \surj I'/({I'}^2T)$. By \cite[4.13]{bk}
 and \cite{d4}, $\psi_1$ can be lifted to a surjection $\Psi : P[T]
\surj I'$.

Applying (\ref{sub}), to $\Psi \ot A(T)$ and $\Phi\ot A(T)$, we get a
surjection $\gD:P(T) \surj IA(T)$ such that $\gD \ot
A(T)/IA(T)=\Phi\ot A(T)/IA(T)=\phi \ot A(T)/IA(T)$.  By (\ref{4.9}),
$\psi$ can be lifted to a surjection $\gT:P[T] \surj I$. This proves
the result.
$\hfill \square$ 
\end{proof}

In case of regular domain, we have the following subtraction and
addition principles.  We give a proof of the
Subtraction principle. The Addition principle can be proved similarly.

\begin{proposition}\label{sub1}
(Subtraction principle) Let $A$ be a regular domain containing an
infinite field $k$ and let $I,J$ be two comaximal ideals of
$A[T]$. Let $P=Q\op A$ be a projective $A$-module of rank $n \geq \dd
A[T]-\hh (I\cap J) +2$. Assume that $\Phi:P[T]\surj I$ and $\Psi :P[T]
\surj I\cap J$ are two surjections such that $\Phi\ot A[T]/I=\Psi\ot
A[T]/I$. Then there exists a surjection $\gD:P[T]\surj J$ such that
$\gD\ot A[T]/J=\Psi \ot A[T]/J$.
\end{proposition}

\begin{proof} 
Let $P_1,\cdots,P_r$ be the associated prime ideals of $I$ and
$Q_1,\cdots,Q_s$ be the associated prime ideals of $J$. As $k$ is
infinite, we can choose $\gl\in k$ such that $T-\gl\notin
(\cup_{1}^{r} P_i)\cup (\cup_{1}^{s}Q_j)$.  If $T-\gl$ is a unit modulo the
ideal $I$, then $I(\gl)=A$; a similar conclusion holds for
$J(\gl)$. In the case when $T-\gl$ is not a unit modulo $I$ or $J$, we
see that $T-\gl$ is a non-zerodivisor modulo $I$ as well as $J$. In
this case, $\hh (I,T-\gl)=\hh I+1$ and $\hh (J,T-\gl)=\hh J+1$. As a
consequence, $\hh I(\gl)\geq \hh I$ and $\hh J(\gl)\geq \hh J$.
Replacing  $T$ by $T-\gl$, we can take $\gl$ to be $0$. 

Note that $\Psi$ induces a surjection $\ol \psi:P[T] \surj J/J^2$. We
have induced surjections $\Phi(0):P\surj I(0)$ and $\Psi(0):P\surj
I(0)\cap J(0)$.

If $J(0)=A$, then $\ol\psi$ can be lifted to a surjection from $ P[T]$
to $ J/(J^2T)$. We now assume $J(0)$ to be a proper ideal.  If
$I(0)=A$, then $\Psi(0)$ is a surjection from $P$ to $J(0)$. Since
$\ol \psi(0)=\Psi(0)\ot A/J(0)$, by (\cite{brs1}, Remark 3.9)
$\ol\psi$ can be lifted to a surjection from $ P[T]$ to $J/(J^2T)$.

Now assume that $I(0),J(0)$ both are proper ideals of $A$. We have,
$\dd A -\hh I(0)\leq \dd A[T]-1-\hh I\leq n-3$.  Similarly, $\dd
(A/J(0))\leq n-3$. Applying Subtraction principle (\ref{sub}) to
surjections $\Phi(0)$ and $\Psi(0)$, we get a surjection $\alpha :
A\surj J(0)$ such that $\alpha\otimes A/J(0)=\Psi(0)\otimes A/J(0) = \ol
\psi(0)$. Consequently, by (\cite{brs1}, Remark 3.9), $\ol\psi$ can be
lifted to a surjection from $ P[T]$ to $ J/(J^2T)$.

Therefore, in any case, $\ol\psi$ can be lifted to a
surjection, say, $\theta : P[T]\surj J/(J^2T)$.

We now go to the ring $A(T)$ and consider surjections $\Phi\otimes
A(T)$ and $\Psi\otimes A(T)$.  Again, applying (\ref{sub}), we can find a
surjection $\beta : P\otimes A(T)\surj JA(T)$ such that $\beta\otimes
A(T)/JA(T)=\Psi\otimes A(T)/JA(T)$. Clearly, $\beta$ lifts
$\gt$. Therefore, by (\ref{4.9}), we get a map $\Delta : P[T]\surj J$
such that $\Delta$ lifts $\theta$.  It is easy to see that $\Delta
\otimes A[T]/J=\Psi\otimes A[T]/J$. This proves the result. 
$\hfill \square$ 
\end{proof}

\begin{proposition}\label{add1}
(Addition principle) Let $A$ be a regular domain containing an infinite field
$k$. Let $I,J$ be two comaximal ideals of $A[T]$ and let $P=Q\op R$
be a projective $A$-module of rank $n\geq \dd A[T] - \hh (I\cap J) +2$.
Let $\Phi:P[T]\surj I$ and $\Psi: P[T]\surj J$ be two
surjections. Then there exists a surjection $\gD:P[T]\surj I\cap J$
such that $\Phi\ot A[T]/I=\gD\ot A[T]/I$ and $\Psi\ot A[T]/J =\gD\ot A[T]/J$.
\end{proposition}


\section{Segre classes}
Let $A$ be a commutative Noetherian ring of dimension $d$ and let
$I\subset A$ be an ideal such that $\mu(I/I^2)=n$ where $n+\hh I\geq
d+3$.  Let $I=(a_1,\cdots,a_n)+I^2$ be given. It is natural to ask
under what condition these local generators can be lifted to a set of
global generators of $I$.  In other words, when can we find
$b_1,\cdots,b_n$ such that $I=(b_1,\cdots,b_n) $ where $a_i-b_i\in
I^2$?

When $\hh I=n$, this has been accomplished in \cite{brs3}, where an
abelian group $E^n(A)$ (called the $n$-th Euler class group of $A$) is
defined and corresponding to the local data for $I$ an element in this
group (called the Euler class) is attached and it is shown that a
desired set of global generators exists for $I$ if the corresponding
Euler class in $E^n(A)$ is zero.

In this section we consider the case when $\hh I$ is not necessarily
equal to $n$.  Given $I$ and $\omega_I:(A/I)^n\surj I/I^2$ (local
data) we shall associate an element $s^n(I,\omega_I)$ in the Euler
class group $E^n(A)$. We call this the \emph{$n$-th Segre class} of
the pair $(I,\omega_I)$. It will be shown that $s^n(I,\omega_I)=0$ in
$E^n(A)$ if and only if $\omega_I$ can be lifted to a surjection
$\theta : A^n\surj I$ (global generators). Further, when $\hh I=n$,
the Segre class coincides with the Euler class of $(I,\omega_I)$.

We may note that the above question was considered in \cite{drs} under
the hypotheses : $d=n=\mu(I/I^2)\geq 3$ and $\hh I\geq 2$.  For
further motivation the reader may look at \cite{drs}, which in turn is
inspired by Murthy's definition of Segre classes \cite{mu} .

Before proceeding to define the Segre class, we first quickly recall
the definition of the $n$-th Euler class group $E^n(A)$ from
\cite{brs3}.

\definition\label{loweuler} 
Let $A$ be a Noetherian ring of dimension
$d$ and let $n$ be an integer with $2n\geq d+3$.  A \emph{local
orientation} $\omega_I$ of an ideal $I\subset A$ of height $n$ is a
surjective homomorphisms from $(A/I)^n$ to $I/I^2$, up to an
$\CE_n(A/I)$-equivalence
(here $\CE_n$ stands for the group of  elementary matrices). 
Let $L^n(A)$ denote the set of all pairs $(I,\omega_I)$, where $I$ is an
ideal of height $n$ such that $\Spec (A/I)$ is connected and $\omega_I:(A/I)^n
\surj I/I^2$ is a local orientation.  
Let $G^n(A)$ denote the free abelian group generated by $L^n(A)$.
Suppose $I$ is an ideal of height $n$ and $\omega_I:(A/I)^n \surj I/I^2$ is a
local orientation. By (\cite{brs3}, Lemma 4.1), there is a unique
decomposition $I=\cap_1^r I_i$, such that $I_i$'s are pairwise
comaximal ideals of height $n$ and $\Spec (A/I_i)$ is connected.  Then
$\omega_I$ naturally induces local orientations 
$\omega_{I_i}:(A/I_i)^n \surj I_i/I_i^2$.
Denote $(I,\omega):=\sum (I_i,\omega_i) \in G^n(A)$. 
We say a local orientation $\omega_I:(A/I)^n\surj I/I^2$ is \emph{global} 
if $\omega_I$ can be lifted to a surjection $\Omega : A^n\surj I$.
Let $H^n(A)$ be the subgroup of $G^n(A)$ generated by all $(I,\omega_I)$ 
where $\omega_I$ is a global orientation. 
The Euler class group of codimension $n$ cycles is defined as
$E^n(A):=G^n(A)/H^n(A)$.


Now let $J$ be an ideal of $A$ such that $J/J^2$ is generated by $n$
elements, where $n+\hh J\geq \dd A+3$. Given a surjection $\omega_J :
(A/J)^n \surj J/J^2$, we will define the $n^{th}$ Segre class of
$(J,\omega_J)$, denoted by $s^n(J,\omega_J)$, as an element of the
$n^{th}$ Euler class group $E^n(A)$, as follows:

\definition\label{seg}
By $(\ref{moving})$, $\omega_J$ can be lifted to a surjection $\ga : A^n
\surj J\cap J_1$, where $J_1$ is an ideal of height $\geq n$ with
$J+J_1=A$. If $J_1=A$, then we define the $n^{th}$ Segre class
$s^n(J,\omega_J)=0$ in $E^n(A)$. If $J_1$ is a proper ideal of height $n$,
then $\ga$ induces a surjection $\omega_{J_1} : (A/J_1)^n \surj
J_1/J_1^2$. We define the $n^{th}$ Segre class
$s^n(J,\omega_J)=-(J_1,\omega_{J_1})$ in $E^n(A)$.

We need to show that $s^n(J,\omega_J)$ is well defined as an element
of $E^n(A)$. The argument is along the same line as in \cite{drs}.  We
give a sketch for the convenience of the reader.

Let $J_{2}$ be another ideal of $A$ of height $\geq n$ such that
$J+J_{2}=A$ and $\omega_J$ has a lift to $\beta:A^n \surj J\cap J_1$.
If $J_2=A$ then it is easy to check using addition and subtraction
principles in the last section that $(J_1,\omega_{J_1})=0$ in $E(A)$.
Therefore assume that $J_2$ is a proper ideal.  Let
$\omega_{J_{2}}:(A/J_{2})^{n}\twoheadrightarrow J_{2}/J_{2}^2$ be the
local orientation induced by $\gb$.  Using Moving lemma \ref{moving}
we can find an ideal $J_{3}$ of $A$ of height $n$ and a local
orientation $\omega_{J_{3}}$ such that : (i) $J_{3}$ is comaximal with
each of $J$, $J_{1}$ and $J_{2}$, (ii)
$(J_{1},\omega_{J_{1}})+(J_{3},\omega_{J_{3}})=0$ in $E(A)$.  Again
applying Lemma \ref{moving} we can find an ideal $J_{4}$ of $A$ of
height $n$ such that $J\cap J_{4}$ is generated by $n$ elements and
$J_{4}$ is comaximal with each of $J$, $J_{1}$, $J_{2}$ and $J_{3}$.

Now addition principle implies that the ideal $J_{1}\cap J_{3}\cap
J\cap J_{4}$ is generated by $n$ elements.  Since $J_{1}\cap J$ is
generated by $n$ elements, by the subtraction principle (
\ref{sub}) it follows that $J_{3}\cap J_{4}$ is generated by $n$
elements with appropriate set of generators.  Now consider $J_{2}\cap
J_{3}\cap J\cap J_{4}$. A similar chain of arguments will show that
$J_{2}\cap J_{3}$ is $n$-generated by the appropriate set of
generators.  Keeping track of the generators, it is easy to see that
this implies $(J_{2},\omega_{J_{2}})+(J_{3},\omega_{J_{3}})=0$ in
$E(A)$.

Therefore, $(J_{1},\omega_{J_{1}})=(J_{2},\omega_{J_{2}})$ in $E(A)$
and $s(J,\omega_{J})$ is well defined.

\remark It is clear from the definition of the Segre class that
$s^n(J,\omega_J)=(J,\omega_J)$ in $E^n(A)$ if $\hh J=n$.

\begin{theorem}
Let $J$ be an ideal of a ring $A$ and let $\omega_J : (A/J)^n \surj J/J^2$
be a surjection, where $n\geq \dd A-\hh J+3$. Suppose $s^n(J,\omega_J)=0$ in
$E^n(A)$. Then $\omega_J$ can be lifted to a surjection $\gt:A^n\surj J$.
\end{theorem}

\proof
Let $s^n(J,w_J)=(J_{1},\omega_{J_{1}})$ in $E^n(A)$ where $\omega_J$ has a lift 
$\ga:A^n \surj J\cap J_1$ and $\omega_{J_1}=\ga \ot A/J_1$.
Now $s^n(J,\omega_{J})=0$ implies $(J_{1},\omega_{J_{1}})=0$ in $E(A)$.
Therefore, by \cite{brs3},
$\omega_{J_{1}}$ is a global orientation of $J_{1}$.
This means that there exist a lift $\gp :A^n\surj J_1$ of $\omega_{J_1}$.
 Now we can apply the subtraction principle (\ref{sub}) to see that $\omega_{J}$ has the 
 desired lift to a surjection $\theta:A^n\surj J$ . This proves the theorem.
$\hfill \square$ 

The following result, on additivity of the Segre classes, is easy and
we leave the proof to the reader.

\begin{theorem}
(Addition) Let $J_1,J_2$ be two comaximal ideals of a ring $A$ and let
$\omega_{J_1}:(A/J_1)^n\surj J_1/J_1^2$ and $\omega_{J_2}:(A/J_2)^n \surj
J_2/J_2^2$ be two surjections, where $n\geq \dd A - \hh (J_1\cap J_2) +3$. 
Then $s^n(J_1\cap J_2,\omega_{J_1\cap
J_2})=s^n(J_1,\omega_{J_1})+s^n(J_2,\omega_{J_2})$ in $E^n(A)$, where
$\omega_{J_1\cap J_2} : (A/(J_1\cap J_2))^n \surj (J_1\cap J_2)/(J_1\cap
J_2)^2$ is the surjection induced by $\omega_{J_1}$ and $\omega_{J_2}$.
\end{theorem}

Let $A$ be a ring of dimension $d$ and $J\subset A$ be an ideal such
that $\mu(J/J^2)=n$ where $n+\hh J\geq d+3$. It is almost a trivial
application of the Nakayama lemma to see that if $A$ is semilocal,
then any $\omega_J:(A/J)^n\surj J/J^2$ can be lifted to a surjection
$\theta:A^n\surj J$. One may wonder if there exists any non-trivial
example of a ring for which such a phenomenon holds. We give one such
below.

\example Let $(A,\mathfrak{m},k)$ be a regular local ring which is
either (i) essentially of finite type over an infinite field; or (ii)
essentially of finite type and smooth over an excellent DVR $(V,\pi)$
such that $k$ is infinite and is separably generated over $V/\pi V$.
Let $\dd A=d+1$ and $f\in \mathfrak{m}\setminus \mathfrak{m}^2$ be a
regular parameter. Let $J\subset A_f$ be an ideal such that
$\mu(J/J^2)=n$ where $n+\hh J\geq d+3$. Then any
$\omega_J:(A_f/J)^n\surj J/J^2$ can be lifted to a surjection
$\theta:A_f^n\surj J$. This follows from \cite[4.2, 5.2]{d4}, since it
is proved there that $E^n(A_f)=0$.

\rmk Let $A$ be a regular domain containing an infinite field $k$ and
let $I\subset A[T]$ be an ideal such that $\mu(I/I^2)=n$, where $n
+\hh I = \text{dim} A[T]+2$. Let $\omega_I:(A[T]/I)^n\surj I/I^2$ be a
given surjection.  Following the same method as in (\ref{seg}), one
can define the $n^{th}$ Segre class $s^n(I,\omega_I)$ as an element of
$E^n(A[T])$ and prove results similar as above using (\ref{sub1},
\ref{add1}).


\section{Homotopy returns}
In the final decade of the last century, Nori suggested a definition
of the ($n$th) Euler class group of a smooth affine domain $R$ of
dimension $n$ and associated for a projective $R$-module $P$ of rank
$n$ (with trivial determinant) an element in this group, called the
Euler class of $P$, and asked whether the vanishing of the Euler class
is the precise obstruction for $P$ to decompose as $P\simeq Q\oplus
R$, for some $R$-module $Q$. In \cite{brs1}, Bhatwadekar-Sridharan
settled Nori's question in the affirmative. They achieved this with a
different (but equivalent to the one proposed by Nori) definition of
the Euler class group which seems to be a bit easier to work
with. (For the record, we may note that \cite[Theorem 3.8]{brs1},
stated below as Theorem \ref{hom}, turned out to be crucial to
establish the equivalence). Moreover, their definition paved the way for 
further generalization to the Euler class group of a Noetherian ring $R$.
All the papers written after \cite{brs1} in this area are based on  the 
definition of Bhatwadekar-Sridharan 
 and within all the development Nori's
definition has been lost. We believe that Nori's definition has its own 
intrinsic appeal and our aim in this section is to investigate it closely.
We first recall how this definition  works for a smooth affine domain (for
which it was formulated). We then show why Nori's definition does not naturally 
extend to singular varieties. Finally, we provide a reformulation
of Nori's definition which extends to general Noetherian rings.

We first quote the following result from \cite{brs1} which is a
special case of Question \ref{open}.  This theorem will be crucially
used to recover results from \cite{brs1} using Nori's definition of
the Euler class group.

\begin{theorem}\cite[3.8]{brs1}\label{hom}
Let $R$ be a smooth affine domain of dimension $n\geq 3$ over an
infinite perfect field $k$. Let $I\subset R[T]$ be an ideal of height
$n$ and $P$ be a projective $R$-module of rank $n$. Assume that 
we are given a surjection $\varphi : P[T]\surj I/(I^2T)$. Then there exists
a surjection $\Phi : P[T]\surj I$ such that $\Phi$ lifts $\varphi$.
\end{theorem}

We now start with both the definitions of Euler class groups as given
in \cite{brs1}.  For simplicity, in this section we consider rings
with dimension $n \geq 3$. Further, as we shall only talk about the
$n^{th}$ Euler class group of a ring $R$ of dimension $n$, we shall
write $E(R)$ instead of $E^n(R)$.

\medskip

\noindent
{\bf Euler class group :} 
\noindent
Let $R$ be a smooth affine domain of dimension $n$ over an infinite
perfect field $k$.  Let $B$ be the set of pairs $(m,\omega_m)$ where
$m$ is a maximal ideal of $R$ and $\omega_m :(R/m)^n\surj m/m^2$. Let
$G$ be the free abelian group generated by $B$.  Let $J=m_{1}\cap
\cdots \cap m_r$, where $m_i$ are maximal ideals of $R$. Any
$\omega_{J}:(R/J)^n\surj J/J^2$ induces surjections $\omega_i
:(R/m_i)^n\surj m_i/m_i^2$ for each $i$. We associate $(J,\omega_J):=
\sum_{1}^{r}(m_i,\omega_i)\in G$.

\definition\label{nori} (Nori) Let $S$ be the set of elements
$(I(1),\omega(1))-(I(0),\omega(0))$ of $G$ where (i) $I\subset R[T]$
is a local complete intersection ideal of height $n$; (ii) Both $I(0)$
and $I(1)$ are reduced ideals of height $n$; (iii) $\omega(0)$ and
$\omega(1)$ are induced by $\omega: (R[T]/I)^n\surj I/I^2$. Let $H$ be
the subgroup generated by $S$. The Euler class group $E(R)$ is defined
as $E(R):=G/H$.

\begin{definition} \label{brs}(Bhatwadekar-Sridharan) Let $H_1$ be the 
subgroup of $G$ generated by those elements $(J,\omega_J)$ of $G$ for
which $\omega_J$ has a lift to a surjection $\theta: R^n\surj J$.  The
Euler class group $E(R)$ is defined as $E(R):=G/H_1$.
\end{definition}

\remark Let $(J,\omega_J)=(I(0),\omega(0))-(I(1),\omega(1))\in S$ and
let $\ol\sigma\in SL_n(R/J)$. Then we have
$\omega_J\ol\sigma:(R/J)^n\surj J/J^2$. Since $\dd R/J=0$, we have
$SL_n(R/J)=\CE_n(R/J)$ where $\CE_n(R/J)$ is the elementary subgroup
of $SL_n(R/J)$. As $\CE_n(R)\lra \CE_n(R/J)$ is surjective, there
exists a preimage $\gs$ of $\ol\gs$. Since $\sigma$ is elementary,
there exists $\Delta\in GL_n(R[T])$ such that $\Delta (0)=I_n$ (the
identity matrix) and $\Delta (1)=\gs$. Let $\omega'=\omega\circ
(\Delta\ot R[T]/I): (R[T]/I)^n\surj I/I^2$.  We note that $J\cap
I(0)=I(1)$ and since all the ideals are reduced, it follows that
$J+I(0)=R$. It is now easy to check that $(J,
\omega_J\ol\sigma)=(I(1),\omega'(1))-(I(0),\omega'(0))$.  Therefore,
$(J, \omega_J\ol\sigma)\in S$.

\remark In Nori's definition, the relations are given by homotopy with
respect to the affine line $\BA^1$.  In this section we shall focus
mainly on relations given by homotopy (hence the title of this
section). We shall revisit some of the results from \cite{brs1}.

\begin{lemma}\cite[4.5]{brs1}\label{ci}
Let $R$ be a smooth affine domain of dimension $n$ and $J\subset R$ be
a reduced ideal of height $n$. Assume that $J=(a_1,\cdots,a_n)$ and
$\omega_J:(R/J)^n\surj J/J^2$ is induced by $a_1,\cdots,a_n$. Then
there exists a local complete intersection ideal $I\subset R[T]$ and a
surjection $\omega:(R[T]/I)^n\surj I/I^2$ such that $I(0),I(1)$ are
both reduced ideals of height $n$ in $R$ and $(J,\omega_J)=
(I(0),\omega(0))-(I(1),\omega(1))$ in $G$.
\end{lemma}

\remark\label{ci1} For a proof, see \cite{brs1} and note that the
above lemma also works for a Noetherian ring $R$ if the terms
``reduced" and ``local complete intersection" are dropped.

In the following important proposition, (\ref{hom}) is crucially
used. Note that (\ref{hom}) does not hold for affine domains which are
not smooth (for an example, see \cite[6.4]{brs1}).

\begin{proposition}\label{homci}
Let $(I(1),\omega(1))-(I(0),\omega(0))\in S$. Then there exists an
ideal $I'\subset R[T]$ of height $n$ such that $I'=(g_1,\cdots,g_n)$
and $
(I(1),\omega(1))-(I(0),\omega(0))=(I'(1),\omega'(1))-(I'(0),\omega'(0))\in
S$ where $\omega':(R[T]/I')^n\surj I'/I'^2$ is induced by
$g_1,\cdots,g_n$.
\end{proposition}

\proof The proposition essentially is a restatement of
\cite[4.3]{brs1} and we shall urge the reader to see \cite{brs1} for
the proof.  Retaining the same notations, we only identify the
terms. Let $\omega:(R[T]/I)^n\surj I/I^2$ be induced by
$I=(f_1,\cdots,f_n)+I^2$.  Using Swan's Bertini theorem instead of
Eisenbud-Evans theorem in Moving lemma, there exists a reduced ideal
$K\subset R$ of height $n$ such that : (1) $K=(b_1,\cdots,b_n)+K^2$;
(2) $I+KR[T]=R[T]$; (3) $I'=I\cap K[T]=(g_1,\cdots,g_n)$ where
$g_i-f_i\in I^2$ and $g_i-b_i\in K^2R[T]$.

Let $\omega':(R[T]/I')^n\surj I'/I'^2$ be induced by $g_1,\cdots,g_n$. It is now easy to 
see that the proposition follows from (1)-(3) above.
$\hfill \square$

\begin{proposition}\label{addsub}
The set $S\subset G$ has the following properties : 
\begin{enumerate}
\item
If $x\in S$, then $-x\in S$.
\item
If $(J_1,\omega_{J_1}),(J_2,\omega_{J_2})\in S$, then
$(J_1,\omega_{J_1})-(J_2,\omega_{J_2})\in S$.
\item
Let $(J_1,\omega_{J_1}),(J_2,\omega_{J_2})\in G$ where $J_1+J_2=R$. If
any two of the elements
$(J_1,\omega_{J_1}),$ $(J_2,\omega_{J_2}),$ $(J_1\cap J_2,\omega_{J_1\cap
J_2})$ belong to $S$, then so does the third.
\end{enumerate}
\end{proposition}

\proof 
(1) Let $x=(I(1),\omega(1))-(I(0),\omega(0))\in S$ where
$I\subset R[T]$ is a local complete intersection ideal of height $n$
and $\omega :(R[T]/I)^n\surj I/I^2$ a surjection. Also $I(0),I(1)$ are
both reduced ideals of height $n$.  Consider the automorphism $\phi:
R[T]\lra R[Y]$ given by $T\mapsto Y=T-1$. Let $\CI=\phi(I)$ and
$\omega':R[Y]/\CI\surj \CI/\CI^2$ be the surjection corresponding to
$\omega$. Then it is easy to see that
$-x=(I(0),\omega(0))-(I(1),\omega(1))=(\CI(1),\omega'(1))-(\CI(0),\omega'(0))\in
S$.

(2) Let $(J_1,\omega_{J_1}),(J_2,\omega_{J_2})\in S$.  Then, there
exists local complete intersection ideals $I_1,I_2\subset R[T]$, each
of height $n$, and surjections $\omega_{i}:(R[T]/I_i)^n \surj I_i/I_i^2$,
$i=1,2$ such that $I_1(0), I_1(1),I_2(0),I_2(1)$ are all reduced
ideals of height $n$ and
$(J_i,\omega_{J_i})=(I_i(1),\omega_i(1))-(I_i(0),\omega_i(0))$ in $G$
for $i=1,2$.

In view of (\ref{homci}), we may assume that $I_1=(f_1,\cdots,f_n)$
and $\omega_1$ is induced by $f_1,\cdots,f_n$. Similarly,
$I_2=(g_1,\cdots,g_n)$ and $\omega_2$ is induced by $g_1,\cdots,g_n$.
Therefore, we have $I_1(0)=(f_1(0),\cdots,f_n(0))$ and $\omega_1(0)$ is
induced by $f_1(0),\cdots,f_n(0)$.  Similar conclusion holds for
$(I_1(1),\omega_1(1))$, $(I_2(0),\omega_2(0))$,
$(I_2(1),\omega_2(1))$.

We have $ J_1\cap I_1(0)=I_1(1)$, $J_2\cap I_2(0)=I_2(1)$. Note that  
$J_1+I_1(0)=R=J_2+I_2(0)$.

Since we have
$(J_i,\omega_{J_i})+(I_i(0),\omega_i(0))=(I_i(1),\omega_i(1))$, it
follows that $f_i(0)-f_i(1)\in I_1(0)^2$. Applying subtraction
principle (\ref{sub}), we conclude that $J_1=(a_1,\cdots,a_n)$ and
$\omega_{J_1}$ is induced by $a_1,\cdots,a_n$.  By a similar argument
it follows that $J_2=(b_1,\cdots,b_n)$ such that $\omega_{J_2}$ is
induced by $b_1,\cdots,b_n$.

Take $\CI=(a_1T+b_1(1-T),\cdots,a_nT+b_n(1-T))\subset R[T]$. Then
$\CI(0)=J_2$ and $\CI(1)=J_1$. Let $\omega_{\CI}:(R[T]/\CI)^n\surj
\CI$ be induced by $a_1T+b_1(1-T),\cdots,a_nT+b_n(1-T)$.  Therefore,
$(J_1,\omega_{J_1})=(\CI(1),\omega(1))$ and
$(J_2,\omega_{J_2})=(\CI(0),\omega(0))$.  This proves the result.

(3) First let $(J_1,\omega_{J_1}),(J_2,\omega_{J_2})\in S$. The proof
goes verbatim as in (2) above, except for the last paragraph. So, we
have $J_1=(a_1,\cdots,a_n)$, $J_2=(b_1,\cdots,b_n)$ and
$J_1+J_2=R$. By addition principle, $J_1\cap J_2=(c_1,\cdots,c_n)$
where $c_i-a_i\in J_1^2$ and $c_i-b_i\in J_2^2$. Further,
$\omega_{J_1\cap J_2}$ is induced by $c_1,\cdots,c_n$. It now follows
from (\ref{ci}) that $(J_1\cap J_2,\omega_{J_1\cap J_2})\in S$.

The other case follows from (2).
$\hfill \square$ 

\medskip


Note that an element of the form $(I(1),\omega(1))-(I(0),\omega(0))$
{\it a priori} need not be of the form $(J,\omega_J)$, where $J\subset
R$ is a reduced ideal. But the above proposition inspires us to
consider the following subset of $G$ :
$$S'=\{(J,\omega_J)\in G\,|, \exists I\subset R[T]\,\,\text{and}\,\,\omega\,\,\text{such that}\,\,(J,\omega_J)=(I(1),\omega(1))-(I(0),\omega(0))\},$$
where the terms have usual meaning. We show 
that it is enough to work with $S'$.

\begin{lemma}\label{equal}
The subgroup of $G$ generated by $S$ is the same as that generated by $S'$.
\end{lemma}

\proof Clearly, from the definition, we have $S'\subset S$. Therefore,
it is enough to show that any $(I(1),\omega(1))-(I(0),\omega(0))\in S$
belongs to the subgroup generated by $S'$.

We may assume by (\ref{homci}) that $I\subset R[T]$ is a complete
intersection, say, $I=(f_1,\cdots,f_n)$ and $\omega$ is induced by
$f_1,\cdots,f_n$. Therefore, we have $I(1)=(f_1(1),\cdots,f_n(1))$,
$I(0)=(f_1(0),\cdots, f_n(0))$ and $\omega(1)$, $\omega(0)$ are
induced by these generators, respectively.

Applying (\ref{ci}) to the element $(I(1),\omega(1))$ of $G$ we can
see that there exists an ideal $\CI\subset R[T]$ of height $n$ and
$\mu:(R[T]/\CI)^n\surj \CI/\CI^2$ such that
$(I(1),\omega(1))=(\CI(1),\mu(1))-(\CI(0),\mu(0))$ in $G$.  Similarly,
we can conclude that
$(I(0),\omega(0))=(\CJ(1),\eta(1))-(\CJ(0),\eta(0))$ in $G$, where
$\CJ\subset R[T]$.  Therefore, $(I(1),\omega(1)),(I(0),\omega(0))\in
S'$ and their difference is in the subgroup generated by $S'$.  $\hfill \square$

\begin{lemma}\label{mov1}
Let $R$ be a smooth affine domain of dimension $n$ over an infinite
perfect field $k$. Let $(J,\omega_J)\in G\setminus S$. Then, there
exists $(K,\omega_K)\in G$ such that $K$ is reduced, $J+K=R$ and
$(J,\omega_J)+(K,\omega_K)\in S$.  Further, given finitely many ideals
$J_1,\cdots,J_l$ of $R$, each of which has height $n$, $K$ can be
chosen to be comaximal with $ J_1\cap \cdots \cap J_l$.
\end{lemma}

\proof Follows from \cite[2.14]{brs2} (using Swan's Bertini theorem
instead of the theorem of Eisenbud-Evans and) and Lemma \ref{ci}
above.

We now state a lemma from \cite[4.13]{k}.

\begin{lemma}\label{mphil}
Let $G$ be a free abelian group with basis $B=(e_i)_{i\in {\mathcal
I}}$ . Let $\sim$ be an equivalence relation on $B$. Define $x\in G$
to be ``reduced" if $x=e_1+\cdots +e_r$ and $e_i\not=e_j$ for $i
\not=j$.
Define $x\in G$ to be ``nicely reduced" if $x=e_1+\cdots +e_r$ is such
that $e_i\not\sim e_j$ for $i \not=j$.  Let $S\subset G$ be such that
\begin{enumerate}
\item
Every element of $S$ is nicely reduced.
\item
Let $x,y\in G$ be such that each of $x,y,x+y$ is nicely reduced. If
two of $x,y,x+y$ are in $S$, then so is the third.
\item
Let $x\in G\setminus S$ be nicely reduced and let ${\mathcal J}\subset
{\mathcal I}$ be finite. Then there exists $y\in G$ with the following
properties : (i) $y$ is nicely reduced; (ii) $x+y\in S$; (iii) $y+e_j$
is nicely reduced $\forall j\in{\mathcal J}$.
\end{enumerate}

Let $H$ be the subgroup of $G$ generated by $S$. If $x \in H$ is
nicely reduced, then $x \in S$.
\end{lemma}

We have the following theorem.

\begin{theorem}\label{zero}
Let $R$ be a smooth affine domain of dimension $n$ over an infinite
perfect field $k$.  Let $(J,\omega_J)=0$ in $E(R)$. Then
$(J,\omega_J)=(I(0),\omega(0))-(I(1),\omega(1))$ in $G$ where (i)
$I\subset R[T]$ is a local complete intersection ideal of height $n$;
(ii) Both $I(0)$ and $I(1)$ are reduced ideals of height $n$; (iii)
$\omega(0)$ and $\omega(1)$ are induced by $\omega: (R[T]/I)^n\surj
I/I^2$.
\end{theorem}

\proof The theorem is a direct application of Lemma \ref{mphil}.  We
take $G$ to be the free abelian group generated by the set $B$ of
pairs $(m,\omega_m)$, as in the definition of the Euler class group at
the beginning of this section. The equivalence relation on $B$ is
simply $(m_1,\omega_{m_1})\sim (m_2,\omega_{m_2})$ if $m_1=m_2$.

We have $E(R)=G/H$, where $H$ is the subgroup generated by $S$ (see
\ref{nori}). By (\ref{equal}) above, $H$ is also generated by $S'$. We
prove this theorem by showing that $S'$ satisfies properties (1)-(3)
of the above lemma.

Let $(J,\omega_J)\in S'$. As $J$ is a reduced ideal, it follows that
$(J,\omega_J)$ is nicely reduced.

Let $(J_1,\omega_{J_1}),(J_2,\omega_{J_2})\in S'$ be nicely reduced
such that $(J_1,\omega_{J_1})+(J_2,\omega_{J_2})$ is also nicely
reduced (i.e., $J_1+J_2=R$). By (\ref{addsub}), if any two of
$(J_1,\omega_{J_1})$, $(J_2,\omega_{J_2})$,
$(J_1,\omega_{J_1})+(J_2,\omega_{J_2})$ belong to $S'$, then so does
the third.  Property (3) follows from (\ref{mov1}).  $\hfill \square$

We now quickly recall the following definition from \cite{brs1}.

\begin{definition}
{\bf Euler class of a projective module :} Let $P$ be a projective
$R$-module of rank $n$ with trivial determinant. Fix $\chi:R\simeq
\wedge^n(P)$. By Swan's Bertini theorem, there is a surjection $\alpha
:P\sur J$ such that $J$ is a reduced ideal of height $n$. Choose an
isomorphism $\sigma:(R/J)^n\simeq P/JP$ (note that $P/JP$ is free by
(\ref{serre})) such that $\wedge^n\sigma=\chi\otimes R/J$.  Let
$\omega_J:(R/J)^n\stackrel{\sigma}\simeq P/JP\stackrel{\ol\alpha}\surj
J/J^2$.  The Euler class of $(P,\chi)$ is defined to be the image of
$(J,\omega_J)$ in $E(R)$ and is denoted as $e(P,\chi)$ (see
\cite{brs1} for further details).
\end{definition}

We reprove the following theorem from \cite{brs1} which shows that
the Euler class is the precise obstruction for a projective module
to split off a free summand.

\begin{theorem}\cite[4.13]{brs1}
Let $R$ be as in (\ref{zero}) and let $P$ be a projective $R$-module
of rank $n$ with trivial determinant. Fix $\chi:R\simeq
\wedge^n(P)$. Then $e(P,\chi)=0$ in $E(R)$ if and only if $P\simeq
Q\oplus R$ for some $R$-module $Q$.
\end{theorem}

\proof 
We first assume that $e(P,\chi)=0$ in $E(R)$.  Let us choose a
surjection $\alpha :P\surj J$, where $J$ is a reduced ideal of height
$n$. Let $\sigma :(R/J)^n\simeq P/JP$ be an isomorphism such that
$\wedge^n\sigma=\chi\otimes R/J$.  Composing with $\alpha\ot R/J$ we
obtain a surjection $\omega_{J}:(R/J)^n\surj J/J^2$. Then
$e(P,\chi)=(J,\omega_J)$.

From the given condition, $(J,\omega_J)=0$ in $E(R)$.  Therefore, by
(\ref{zero}), we have $(J,\omega_J)=(I(1),\omega(1))-(I(0),\omega(0))$
in $G$ where (i) $I\subset R[T]=(f_1,\cdots,f_n)$ is a complete
intersection ideal of height $n$; (ii) Both $I(0)$ and $I(1)$ are
reduced ideals of height $n$; (iii) $\omega(0)$ and $\omega(1)$ are
induced by $\{f_1(0),\cdots,f_n(0)\}$ and $\{f_1(1),\cdots,f_n(1)\}$,
respectively.

For simplicity, let us write $f_i(1)=a_i$ for $i=1,\cdots,n$ and
$f_i(0)=b_i$ for $i=1,\cdots,n$.  As
$(J,\omega_J)+(I(0),\omega(0))=(I(1),\omega(1))$ in $G$ and all the
ideals are reduced, it follows that $J\cap I(0)=I(1)$ and
$J+I(0)=R$. Now, $I(0)=( b_1,\cdots,b_n)$. Using a standard general position 
argument we may assume that $\text{ht}(b_1,\cdots,b_{n-1})=n-1$ and
$(b_1,\cdots,b_{n-1})+J=R$. We consider
$I'=(b_1,\cdots,b_{n-1},(1-b_n)T+b_n)\subset R[T]$ and $I''=I'\cap
J[T]$. Let $\omega'':(R[T]/I'')^n\surj I''/I''^2$ be induced by
$\omega_J$ and $\{b_1,\cdots,b_{n-1},(1-b_n)T+b_n\}$.  Then,
$I''(1)=J$, $I''(0)=J\cap I(0)=I(1)$ and moreover,
$(I''(1),\omega''(1))=(J,\omega_J)$ and
$(I''(0),\omega''(0))=(I(1),\omega(1))$ in $G$.

We have $e(P,\chi)= (J,\omega_J)=(I''(1),\omega''(1))$. Therefore, by
(\ref{hom}) and \cite[Theorem, pp 457]{mrs} it follows that there
exists a surjection $\beta : P\surj I(1)$ such that
$e(P,\chi)=(I(1),\omega(1))$. Combination of (\ref{hom}) and
\cite[Theorem, pp 457]{mrs} essentially implies that if two Euler
cycles are homotopic and one of them is the Euler class of a
projective module, then the other one is also the Euler class of the
same projective module.

Note that, we have $I(1)=(a_1,\cdots,a_n)$ and $\omega(1)$ is induced
by $\{a_1,\cdots,a_n\}$.  By a general position argument we may assume
that $\hh(a_1,\cdots,a_{n-1})=n-1$. Let
$K=(a_1,\cdots,a_{n-1},T^2-Ta_n+a_n)$.  Let $\gamma: R[T]^n\surj K$ be
the map which is given by $e_i\mapsto a_i$ for $1\leq i\leq n-1$ and
$e_n\mapsto T^2-Ta_n+a_n$.

We choose  $\Delta :P[T]/KP[T]\simeq (R[T]/K)^n$ such
that $\wedge^n\Delta=(\chi\ot R[T]/K)^{-1}$.  Composing, we get a
surjection $\ol\phi:P[T]\surj K/K^2$. It is easy to check that
$\ol\phi(0)=\beta\ot R/I(1)$.  As $K(0)=I(1)$, it follows from
\cite[3.9]{brs1} that $\ol\phi$ has a lift to $\ol\varphi:P[T]\surj
K/(K^2T)$. By (\ref{hom}) $\ol\varphi$ can be lifted to $\varphi:
P[T]\surj K$. But then $\varphi(1):P\surj K(1)=R$.

Conversely, if $P\simeq Q\op R$, by a result of Mohan Kumar
\cite[Theorem 1]{mk2}, $P$ maps onto an ideal $J$ of height $n$ such
that $J$ is generated by $n$ elements. It is easy to check using
(\ref{homci}) that $e(P,\chi)=0$ in $E(R)$.  $\hfill \square$

\bigskip

{\bf From now on we assume $R$ to be a  Noetherian ring of
dimension $n$ containing $\bq$}.  For such a ring the ($n$-th) Euler
class group was defined in \cite{brs2}. Let us quickly recall.

\definition\label{noeth} Let $G$ be the free abelian group on the set
$B$ of pairs $(\mathfrak{n},\omega_{\mathfrak{n}})$, where
$\mathfrak{n}$ is an $\mathfrak{m}$-primary ideal of height $n$ such
that $\mu(\mathfrak{n}/\mathfrak{n}^2)=n$ and
$\omega_{\mathfrak{n}}:(R/\mathfrak{n})^n\surj
\mathfrak{n}/\mathfrak{n}^2$ is a surjection. Let $J\subset R$ be an
ideal of height $n$ such that $\mu(J/J^2)=n$ and let
$\omega_J:(R/J)^n\surj J/J^2$ be a surjection. Let
$J=\cap_{1}^{r}\mathfrak{n}_i$ be the (irredundant) primary
decomposition of $J$. Therefore, $\mathfrak{n}_i$ is an
$\mathfrak{m}_i$-primary ideal for some maximal ideal $\mathfrak{m}_i$
of height $n$. Then, $\omega_J$ induces
$\omega_{\mathfrak{n}_i}:(R/\mathfrak{n}_i)^n\surj
\mathfrak{n}_i/\mathfrak{n}_i^2$. One associates and writes
$(J,\omega_J):=\sum_{1}^{r}
(\mathfrak{n}_i,\omega_{\mathfrak{n}_i})\in G$. Let $H_1$ be the
subgroup of $G$ generated by elements $(J,\omega_J)$ for which
$\omega_J$ has a lift to a surjection $\theta:R^n\surj J$. The Euler
class group is defined as $E(R):=G/H_1$.

Compare this definition with (\ref{brs}). Now, if we want to extend
Nori's definition (\ref{nori}) to a Noetherian ring, we
should first ask, obviously, the following natural question :

\quest\label{qn} Let $R$ be a commutative Noetherian ring of dimension
$n$ containing $\bq$.  Let $I\subset R[T]$ be an ideal of height $n$
such that there is a surjection $\omega:(R[T]/I)^n\surj I/I^2$. Assume
further that $I(0)$ and $I(1)$ are both of height $n$.  Then, do we
have $(I(1),\omega(1))=(I(0),\omega(0))$ in $E(R)$?

To put the above question in proper perspective, we proceed in the
 following way.  Let $K$ be the subset of $E(R)$ which consists of
 elements $(J,\omega_J)$ for which $
 (J,\omega_J)=(I(1),\omega(1))-(I(0),\omega(0))$ in $E(R)$ for some
 ideal $ I\subset R[T]$ and surjection $ \omega:(R[T]/I)^n\surj
 I/I^2$.  We now prove the following proposition.

\begin{proposition}
The set $K$, as described above, is a subgroup of $E(R)$.
\end{proposition}

\proof
It is easy to check that if $x\in K$, then $-x\in K$. 

Let $(J,\omega_J)=(I(1),\omega(1))-(I(0),\omega(0))$ in $E(R)$ and
$(J',\omega_J')=(I'(1),\omega'(1))-(I'(0),\omega'(0))$ in $E(R)$.  We
need only show that $(J,\omega_J)+(J',\omega_J')\in K$.  For this
proof we need to invoke a result from \cite{d2}.  Applying the moving
lemma (\ref{moving}) twice, we can find $(\CI,\omega_{\CI})\in
E(R[T])$ such that $(I,\omega)=(\CI,\omega_{\CI})$ in $E(R[T])$ and
$\CI+I'=R[T]$.  Now, there exists a group homomorphism $\Phi_0:
E(R[T])\surj E(R)$ which is induced by specialization at $T=0$. It
takes $(I,\omega)$ to $(I(0),\omega(0))$ and $(\CI,\omega_{\CI})$ to
$(\CI(0),\omega_{\CI(0)})$. Similarly, one can define a group
homomorphism at $T=1$ with similar properties. Therefore,
$(I(1),\omega(1))-(I(0),\omega(0))=(\CI(1),\omega_{\CI(1)})-(\CI(0),\omega_{\CI(0)})$
in $E(R)$.  Now,
$$(J,\omega_J)+(J',\omega_J')=(\CI(1),\omega_{\CI(1)})-(\CI(0),\omega_{\CI(0)})+(I'(1),\omega'(1))-(I'(0),\omega'(0))$$
and as $\CI+I'=R[T]$, writing $I''=I\cap \CI$ we have,
$$(J,\omega_J)+(J',\omega_J')=(I''(1),\tilde{\omega}(1))-( I''(0),\tilde{\omega}(0)),$$
where $\tilde{\omega}:(R[T]/I'')^n\surj I''/I''^2$ is induced by $\omega_{\CI}$
and $\omega'$.
$\hfill \square$ 

\medskip

We may now rephrase Question \ref{qn} in the following way.

\medskip

\quest
Is $K$ a non-trivial subgroup of $E(R)$?

\medskip

\rmk It is clear from the first half of this section that if $R$ is a
smooth affine domain over an infinite perfect field, then $K$ is
trivial.

Bhatwadekar, through personal communication, pointed out to us that if
$R$ is not smooth, then $K$ could be non-trivial, even for a normal
affine algebra over an algebraically closed field. His example is as
follows. We sincerely thank him for allowing us to include the example here.

\begin{example}\label{smb}(Bhatwadekar)
We consider the same affine algebra as in \cite[Example 6.4]{brs1}.
We shall freely use facts and details from that example. Let
$$B=\frac{\BC [X,Y,Z,W]}{(X^5+Y^5+Z^5+W^5)}$$ 
Then $B$ is a graded
normal affine domain over $\BC$ of dimension $3$, having an isolated
singularity at the origin. Let $F(B)$ be the subgroup of $\wt{K}_0(B)$
generated by all elements of the type $[P]-[P^\ast]$, where $P$ is a
finitely generated projective $B$-module. As $B$ is graded,
$\text{Pic}(B)=0$. Therefore, by \cite[6.1]{brs1}
$F(B)=F^3K_0(B)$. Since $\text{Proj}(B)$ is a smooth surface of degree
$5$ in $\BP^3$, it follows from a result of Srinivas that $F(B)=
F^3K_0(B)\not = 0$. Therefore, there exists a projective $B$-module
$P$ of rank $3$ with trivial determinant such that $[P] - [P^{\ast}]$
is a nonzero element of $F(B)$.  This implies that $P$ does not have a
unimodular element.  We now consider the ring homomorphism $f: B
\rightarrow B[T]$ given by $f(x) = xT, f(y) = yT, f(z) = zT, f(w) =
wT$.  We regard $B[T]$ as a $B$-module through this map and $Q = P
\otimes_B B[T]$.  Then it is easy to see that $Q/TQ$ is free and
$Q/(T-1)Q = P$.  Therefore, $Q$ is a projective $B[T]$-module which is
not extended from $B$. Now consider a surjection $\alpha : Q\surj I$
where $I\subset B[T]$ is an ideal of height $3$. Fix an isomorphism
$\chi:B[T]\simeq \wedge^3(Q)$. Note that, by (\ref{serre}), $Q/IQ$ is a free
$B[T]/I$-module. We choose an isomorphism $\sigma:(B[T]/I)^3\simeq
Q/IQ$ such that $\wedge^3\sigma= \chi\otimes B[T]/I$. Composing
$\sigma$ and $\alpha\otimes B[T]/I$, we obtain a surjection
$\omega:(B[T]/I)^3\surj I/I^2$.  It is now easy to see that
$(I(0),\omega(0))=0$ in $E(B)$ (as $Q/TQ$ is free), whereas
$(I(1),\omega(1))=e(Q/(T-1)Q,\chi(1))=e(P,\chi(1))$ in $E(B)$ cannot
be trivial (as $P$ does not have a unimodular element).
$\hfill \square$ 
\end{example}

\bigskip

We now extend  Nori's definition to Noetherian rings, as follows.

\begin{definition}\label{noeth1}
Let $R$ be a commutative Noetherian ring of dimension $n$ containing
$\bq$. Let $G$ be the free abelian group as considered in (\ref{noeth}).  Now
let $S$ be the set of elements
$(J,\omega_J)=(I(0),\omega(0))-(I(1),\omega(1))$ of $G$ where (i)
$I\subset R[T]$ is an ideal of height $n$ such that
$I=(f_1,\cdots,f_n)$; (ii) Both $I(0)$ and $I(1)$ are ideals of height
$n$; (iii) $\omega(0)$ and $\omega(1)$ are induced by
$\{f_1(0),\cdots,f_n(0)\}$ and $\{f_1(1),\cdots,f_n(1)\}$,
respectively.  Let $H$ be the subgroup generated by $S$. We define
the group $\wt{E}(R)$ as $\wt{E}(R):=G/H$.
\end{definition}

Clearly, as an easy consequence of Lemma \ref{homci} above, one can
see that when $R$ is a smooth affine domain, this definition is
equivalent to the one given in \ref{nori}. We now show that this
definition is equivalent to the definition given by
Bhatwadekar-Sridharan for Noetherian rings in \cite{brs2}
((\ref{noeth}) here).

\begin{proposition}
Let $E(R)$ be the Euler class group defined as in (\ref{noeth}). Then
$\wt{E}(R)\simeq E(R)$.
\end{proposition}

\proof It is obvious from definitions (\ref{noeth}, \ref{noeth1}) that
$H\subset H_1$.  To see that $H_1\subset H$, apply (\ref{ci},
\ref{ci1}).  $\hfill \square$

\remark In a similar manner, one can also define the weak Euler class
group $E_{0}(R)$ in terms of homotopy.

\noindent{\bf Acknowledgements:} We sincerely thank S. M. Bhatwadekar
for Example \ref{smb}.

{}

\end{document}